\documentclass{article} 
\usepackage[utf8]{inputenc}
\usepackage[margin =2.4cm]{geometry}
\usepackage{amssymb}
\usepackage{amsmath}
\usepackage{amsthm}
\usepackage{enumerate} 
\usepackage{graphicx}
\usepackage{csquotes}
\usepackage{subcaption}
\usepackage{authblk}
\usepackage{tikz}

\newtheorem{theorem}{Theorem}

\newtheorem{corollary}[theorem]{Corollary}

\newtheorem{example}{Example}

\title{\LARGE{Mixed Minkowski-Covering Inequalities for \\ Convex Bodies and Lattices}}

\author{
    {\Large Aled Williams}\\
    Department of Computer Science\\
    Royal Holloway, University of London\\
    London, UK\\
    \texttt{aled.williams@rhul.ac.uk}
}

\begin{document}
	\date{}
	\maketitle

\begin{abstract}
In this paper we present a sharp mixed inequality relating successive minima and quotient covering radii of origin-symmetric convex bodies with respect to lattices. The inequality interpolates between the (classical) covering-density lower bound and the lower bound in Minkowski's second theorem. 

\vspace{1.0mm}
\end{abstract}


\section{Introduction and Preliminaries}
A convex set $K \subset \mathbb{R}^n$ which is compact, i.e. closed and bounded, and has a nonempty interior is called a convex body. 
Recall that a set is convex if the line segment joining any two of its points lies entirely in the set.
We denote by $\mathcal{K}^n$ the family of all convex bodies in $\mathbb{R}^n$. 
A convex body $K$ is called origin-symmetric (or $o$-symmetric) if it is centrally symmetric with respect to the origin, i.e. $\boldsymbol{x} \in K$ if and only if $-\boldsymbol{x} \in K$.  We denote by $\mathcal{K}^n_0$ the family of all $o$-symmetric convex bodies in $\mathbb{R}^n$. 

For linearly independent $\boldsymbol{b}_1, \ldots, \boldsymbol{b}_k \in \mathbb{R}^n$, the set
$$
\Lambda
=
\left\{
\sum_{i=1}^{k} x_i \boldsymbol{b}_i
:
x_i \in \mathbb{Z}
\right\}
$$
is a $k$-dimensional lattice generated by the basis
$\boldsymbol{b}_1,\ldots,\boldsymbol{b}_k$.  
If
$B = [\boldsymbol{b}_1,\ldots,\boldsymbol{b}_k] \in \mathbb{R}^{n \times k}$
is a matrix with columns $\boldsymbol{b}_1,\ldots,\boldsymbol{b}_k$, then the lattice $\Lambda = \Lambda(B)$
can be expressed equivalently as
$\Lambda = \left\{B \boldsymbol{x} : \boldsymbol{x} \in \mathbb{Z}^k\right\}.$ If $n=k$, then $\Lambda$ is a full-rank lattice. We denote by $\mathcal{L}^n$ the family of all full-rank lattices in $\mathbb{R}^n$.

For $K \in \mathcal{K}^n_0$ and $\Lambda \in \mathcal{L}^n$, the $i$-th \textit{successive minimum} $\lambda_i(K,\Lambda)$ of $K$ with respect to the lattice $\Lambda$, for $i = 1,\ldots,n$, is defined by 
$$
\lambda_i(K, \Lambda) = \min \left\{ \lambda > 0 : \operatorname{dim}(\lambda K \cap \Lambda) \ge i \right\}, 
$$
where $\operatorname{dim}(A)$ for $A \subset \mathbb{R}^n$ is the dimension of the vector space spanned by all vectors of the set $A$. In other words, the $i$-th successive minimum $\lambda_i(K,\Lambda)$ is the smallest (strictly positive) dilation factor $\lambda$ such that $\lambda K$ contains at least $i$ linearly independent lattice vectors.
Note that
\begin{equation} \label{ordering_successive}
\lambda_1(K,\Lambda)\le \cdots \le \lambda_n(K,\Lambda).
\end{equation}

\begin{example}
Let $K=[-1,1]^n$ and $\Lambda=\mathbb Z^n$, then
$\lambda_1(K,\mathbb Z^n)=\cdots=\lambda_n(K,\mathbb Z^n)=1$. This follows since $K$ already contains the $n$ standard/elementary basis vectors
$\boldsymbol{e}_1,\ldots,\boldsymbol{e}_n$, while for any positive dilation factor $\lambda<1$, the set $\lambda K$ contains no nonzero lattice points.
\end{example}

For $K \in \mathcal{K}^n$ and $\Lambda \in \mathcal{L}^n$, we define the \textit{covering radius} $\mu(K, \Lambda)$ of $K$ with respect to the lattice $\Lambda$ as 
$$\mu(K, \Lambda) = \inf\{\mu > 0 : \mu K + \Lambda 
= \mathbb{R}^n
\}.$$
Thus, the covering radius $\mu(K, \Lambda)$ is the smallest scale factor $\mu$ such that translates of $\mu K$ by lattice vectors from $\Lambda$ cover $\mathbb{R}^n$. 

\begin{example}
If $K=[-1,1]^n$ and $\Lambda=\mathbb{Z}^n$, then 
$\mu(K,\mathbb{Z}^n)=\frac12$, which follows since the translates of $\frac12K=[-\frac{1}{2},\frac12]^n$ by integer lattice points cover $\mathbb{R}^n$.
\end{example}

\begin{example}
Let $B_2^n = \{ \boldsymbol{x} \in \mathbb{R}^n : \| \boldsymbol{x} \|_2 \le 1\}$ denote the unit $\ell_2$ (Euclidean) ball in $\mathbb{R}^n$ and $\Lambda=\mathbb{Z}^n$, then 
$\mu(B_2^n,\mathbb{Z}^n) = \frac{\sqrt n}{2}$. This follows since every point in $\mathbb{R}^n$ is congruent modulo $\mathbb{Z}^n$ to a point in the $o$-symmetric unit cube $[-\frac12,\frac12]^n$, and the farthest points of this cube from the origin have $\ell_2$-norm $\sqrt n/2$.
\end{example}

In order to define the determinant of a lattice we will firstly introduce the \textit{fundamental parallelepiped} of a given lattice. Recall that $\Lambda=\Lambda(B)$ denotes the lattice generated by the basis $B = [\boldsymbol{b}_1,\ldots,\boldsymbol{b}_k] \in \mathbb{R}^{n \times k}$. The fundamental parallelepiped of the lattice $\Lambda$ associated with basis $B$ is 
$$
\mathcal P(B)
=
\{B\boldsymbol x:\boldsymbol x\in\mathbb R^k,\ 0\le x_i<1
\text{ for } i=1,\ldots,k\}.
$$

The determinant of $\Lambda$ is the $k$-dimensional volume of the fundamental parallelepiped $\mathcal P(B)$, which is calculated \cite{gover2010determinants} by
$$
\det(\Lambda)=\operatorname{vol}_{k}(\mathcal{P}(B))=\sqrt{\det\left(B^{T} B\right)}\,,
$$
where $B^T$ denotes the transpose of $B$ and $\operatorname{vol}_{k}(\cdot)$ denotes the $k$-dimensional volume or Lebesgue measure (see e.g. \cite[Chapter 13]{bartle1995elements}). If $\Lambda(B)$ is a full rank lattice (where $n = k$)
then the matrix $B$ is a square matrix and the determinant of the lattice is given by $\det(\Lambda) = | \det(B) |$, where $| \cdot |$ denotes the absolute value. 

Let \(K\in\mathcal K^n_0\) and \(\Lambda\in\mathcal L^n\). We now recall the two endpoint inequalities which motivate the main result. The first is the elementary covering-density lower bound
\begin{equation} \label{cover_density_lower}
\mu(K,\Lambda)^n \operatorname{vol}_n(K)\ge \det(\Lambda).
\end{equation}
For any $t \ge \mu(K,\Lambda)$, the translates $tK+\Lambda$ cover $\mathbb{R}^n$. Hence their lattice covering density ${\operatorname{vol}_n(tK)}/{\det(\Lambda)}$ is at least 1 (see e.g. \cite{rogers1959lattice}), and so
$$
\frac{\operatorname{vol}_n(tK)}{\det(\Lambda)}
=
\frac{t^n\operatorname{vol}_n(K)}{\det(\Lambda)}
\ge 1.
$$
Thus, setting $t = \mu(K,\Lambda)$ yields \eqref{cover_density_lower} as required. 

The second endpoint inequality is the lower bound in Minkowski's second theorem on successive minima (see e.g. \cite[Chapter 9]{gruber1987geometry}). The full theorem states that
$$
\frac{2^n}{n!}\det(\Lambda)
\le
\left(\,\prod_{i=1}^n \lambda_i(K,\Lambda)\right)
\operatorname{vol}_n(K)
\le
2^n\det(\Lambda).
$$
In particular, the motivating inequality is
\begin{equation}\label{minkowski_lower}
\prod_{i=1}^n \lambda_i(K,\Lambda)
\ge
\frac{2^n}{n!}
\frac{\det(\Lambda)}{\operatorname{vol}_n(K)}.
\end{equation}

The result below interpolates between \eqref{cover_density_lower} and \eqref{minkowski_lower} by replacing the last $n-r$ successive minima with a quotient covering radius.


\section{Main Results}
We now state the main result. The statement is most naturally formulated in terms of a quotient (or projected) covering radius, which we define below. The corresponding inequalities for the ordinary covering radius and for covering minima follow as immediate consequences, although they are in general weaker.

Kannan and Lov{\'a}sz \cite{KannanLovasz} introduced covering minima, and several Minkowski-type inequalities for them have since been studied (see e.g. \cite{BetkeHenkWills,schnell1995minkowski,merino2017densities}). The result below is close in spirit to mixed inequalities relating successive minima and radii, such as those in \cite{henk2009successive}, however, the transverse term depends on a quotient covering radius. Related projection-based interpolation phenomena occur in the theory of packing minima \cite{HenkSchymuraXue}. To the best of our knowledge, the mixed successive-minima/projected-covering inequality below has not appeared previously.

Let $K\in\mathcal K^n_0$ and let $\Lambda\in\mathcal L^n$. For $0\le r\le n$, choose lattice vectors
$v_1,\ldots,v_r\in\Lambda$ recursively such that
$$
v_i\in \lambda_i(K,\Lambda)K
\quad\text{and}\quad
v_1,\ldots,v_i \text{ are linearly independent}
$$
for every \(1\le i\le r\). Note that such a selection is possible by the definition of the successive minima. Let
$$
L_r=\operatorname{lin}(v_1,\ldots,v_r) \subset \mathbb R^n,
$$
where $\operatorname{lin}(A)$ denotes the linear span of $A$, with the convention that $L_0 = \{0\}$.

Let
$\pi_{L_r}:\mathbb R^n\to L_r^\perp$ denote the orthogonal projection. We define 
$$
K_{L_r}=\pi_{L_r}(K)
$$
to be the projected convex body,
$$
\Lambda_{L_r}=\pi_{L_r}(\Lambda)
$$
to be the quotient (or projected) lattice, and finally set 
$$
\rho_{L_r}=\mu(K_{L_r},\Lambda_{L_r}) \, .
$$
If $r=n$, then the projected space is zero-dimensional (i.e. $L_r^\perp=\{0\}$), and in this case $\rho_{L_r}^{n-r}=\rho_{L_r}^{0}$ is by convention equal to $1$. We also use the standard convention that an empty product is equal to $1$.

The main result is the following.

\begin{theorem}[Quotient form]\label{thm:main}
If $K\in\mathcal K^n_0$ and $\Lambda\in\mathcal L^n$, then for every $0\le r\le n$, we have
\begin{equation} \label{theorem_1_eq}
\left( \, \prod_{i=1}^r \lambda_i(K,\Lambda)\right)
\rho_{L_r}^{\,n-r}
\ge
\frac{2^r (n-r)!}{n!}
\frac{\det(\Lambda)}{\operatorname{vol}_n(K)}.
\end{equation}
\end{theorem}

Theorem~\ref{thm:main} immediately implies a version stated only in terms of the covering radius. Recall that if $t \ge \mu(K,\Lambda)$, then $tK+\Lambda=\mathbb R^n$.
Upon applying the projection $\pi_{L_r}$, we obtain
$$
tK_{L_r}+\Lambda_{L_r}=L_r^\perp.
$$
In particular, the projected convex body $K_{L_r}$, dilated by the same factor $t$, covers $L_r^\perp$ under translations by the quotient lattice $\Lambda_{L_r}$. Hence, it follows that
$$
\rho_{L_r}=\mu(K_{L_r},\Lambda_{L_r})\le t.
$$
Since this holds for every $t\ge\mu(K,\Lambda)$, we obtain
$$
\rho_{L_r}\le \mu(K,\Lambda).
$$
Therefore, Theorem~\ref{thm:main} gives the following corollary.

\begin{corollary}[Covering-radius form]\label{cor:covering-radius}
If $K\in\mathcal K^n_0$ and $\Lambda\in\mathcal L^n$, then for every $0\le r\le n$, we have
\begin{equation} \label{corlloary_covering_radius_form}
\left(\prod_{i=1}^r \lambda_i(K,\Lambda)\right)
\mu(K,\Lambda)^{\,n-r}
\ge
\frac{2^r (n-r)!}{n!}
\frac{\det(\Lambda)}{\operatorname{vol}_n(K)}.
\end{equation}
\end{corollary}

We also record a formulation in terms of covering minima. In particular, for $1\le d\le n$, the $d$-th covering minimum of $K$ with respect to $\Lambda$ is defined by
$$
\mu_d(K,\Lambda)
=
\max_{\pi}\mu(\pi (K),\pi(\Lambda)),
$$
where the maximum is taken over all linear projections $\pi:\mathbb R^n\to\mathbb R^d$ for which $\pi(\Lambda)$ is a lattice (see e.g. \cite{KannanLovasz,CodenottiSantosSchymura}). Since $\rho_{L_r}$ is one of the quotient covering radii appearing in this maximum, we have
$$
\rho_{L_r}\le \mu_{n-r}(K,\Lambda).
$$
Therefore, Theorem~\ref{thm:main} yields the following corollary.

\begin{corollary}[Covering-minimum form]\label{cor:covering-minima}
If $K\in\mathcal K^n_0$ and $\Lambda\in\mathcal L^n$, then for every $0\le r<n$, we have
$$
\left(\prod_{i=1}^r \lambda_i(K,\Lambda)\right)
\mu_{n-r}(K,\Lambda)^{\,n-r}
\ge
\frac{2^r (n-r)!}{n!}
\frac{\det(\Lambda)}{\operatorname{vol}_n(K)}.
$$
\end{corollary}

Note that if $r=0$, then Corollary~\ref{cor:covering-radius} is precisely the usual covering-density lower bound. If $r=n$, then Theorem~\ref{thm:main} is precisely the lower bound in Minkowski's second theorem on successive minima.


\section{Proof of Theorem~\ref{thm:main}}

\begin{proof}
Fix $0\le r\le n$. In order to simplify notation slightly, throughout this proof we let
$$
P_r=\prod_{i=1}^r \lambda_i(K,\Lambda),
$$
and write 
$$
L=L_r=\operatorname{lin}(v_1, \ldots, v_r) \subset \mathbb R^n.
$$ 
The cases $r=0$ and $r=n$ are included in the argument below, following the usual zero-dimensional conventions.

Let $\Gamma=\Lambda\cap L$. Since $v_1,\ldots,v_r\in\Gamma$ and these vectors are linearly independent by construction, $\Gamma$ is a lattice of rank $r$ in $L$. Moreover, for each $1\le i\le r$, the vectors $v_1,\ldots,v_i$ are linearly independent. Recalling \eqref{ordering_successive},
we have
$$
v_1,\ldots,v_i\in \lambda_i(K,\Lambda)K, 
$$
and since these vectors also lie in $L$, we obtain
$$
v_1,\ldots,v_i\in \lambda_i(K,\Lambda)(K\cap L).
$$
Thus, $\lambda_i(K,\Lambda)(K\cap L)$ contains $i$ linearly independent vectors of $\Gamma$. Hence, by the definition of the successive minima of $K\cap L$ with respect to $\Gamma$, it follows that
\begin{equation} \label{successive_minima_restrictive_ordering}
\lambda_i(K\cap L,\Gamma)\le \lambda_i(K,\Lambda)
\end{equation}
for each $1 \le i \le r$.

Using the lower bound in Minkowski's second theorem on successive minima \eqref{minkowski_lower}, applied within the $r$-dimensional subspace $L$, gives
\begin{equation} \label{minkowski_restricted}
\prod_{i=1}^r \lambda_i(K\cap L,\Gamma)
\ge
\frac{2^r}{r!}
\frac{\det(\Gamma)}{\operatorname{vol}_r(K\cap L)}.
\end{equation}
Recalling that $P_r=\prod_{i=1}^r \lambda_i(K,\Lambda)$ and using the inequality \eqref{successive_minima_restrictive_ordering}, we deduce from \eqref{minkowski_restricted} that
\begin{equation}\label{eq:section-det-bound}
\det(\Gamma)
\le
\frac{r!}{2^r}
\operatorname{vol}_r(K\cap L)P_r .
\end{equation}

Consider the projected convex body $K_L=\pi_L(K)$ and the quotient lattice $\Lambda_L=\pi_L(\Lambda)$ in $L^\perp$. Using the elementary covering-density lower bound \eqref{cover_density_lower}, applied within the $(n-r)$-dimensional subspace $L^\perp$, we have 
\begin{equation}\label{eq:quotient-covering-bound}
\rho_L^{\,n-r}\operatorname{vol}_{n-r}(K_L)
\ge
\det(\Lambda_L),
\end{equation}
where we recall that $\rho_L = \mu(K_L, \Lambda_L)$ is the covering radius of $K_L$ with respect to $\Lambda_L$.

Noting that $L$ is spanned by $r$ linearly independent lattice vectors, it follows that $\Gamma=\Lambda\cap L$ is a full-rank lattice in the $r$-dimensional subspace $L$. The standard determinant factorisation for a lattice and its orthogonal projection (see e.g. \cite[Proposition~1.9.7]{martinet2003perfect}) gives
\begin{equation}\label{eq:det-factorisation}
\det(\Lambda)
=
\det(\Gamma)\det(\Lambda_L). 
\end{equation}

Upon combining the inequalities \eqref{eq:section-det-bound} and \eqref{eq:quotient-covering-bound}, with the determinant factorisation \eqref{eq:det-factorisation}, we obtain
\begin{equation} \label{volume_product_bound}
\det(\Lambda)
\le
\frac{r!}{2^r}
P_r\rho_L^{\,n-r}
\operatorname{vol}_r(K\cap L)\operatorname{vol}_{n-r}(K_L).
\end{equation}

Finally, by the Rogers--Shephard section-projection inequality \cite[Theorem~1]{rogers1958convex}, we have
$$
\operatorname{vol}_n(K)
\ge
\frac{r!(n-r)!}{n!}
\operatorname{vol}_r(K\cap L)\operatorname{vol}_{n-r}(K_L),
$$
or equivalently
$$
\operatorname{vol}_r(K\cap L)\operatorname{vol}_{n-r}(K_L)
\le
\frac{n!}{r!(n-r)!}\operatorname{vol}_n(K)
=
\binom{n}{r}\operatorname{vol}_n(K).
$$

Upon substituting this upper bound into \eqref{volume_product_bound}, we obtain
$$
\begin{aligned}
\det(\Lambda)
&\le
    \frac{r!}{2^r}
    P_r\rho_L^{\,n-r}
    \binom{n}{r}\operatorname{vol}_n(K) \\
&= \frac{r!}{2^r}
    P_r\rho_L^{\,n-r} \frac{n!}{r!(n-r)!}\operatorname{vol}_n(K) \\
&= {P_r\rho_L^{\,n-r}}
     \frac{n!}{2^r(n-r)!}\operatorname{vol}_n(K), 
\end{aligned}
$$
which rearranges to \eqref{theorem_1_eq} as required, and concludes the proof of Theorem~\ref{thm:main}.
\end{proof}


\section{Sharpness and Comparison}

\subsection{Endpoint Sharpness}
The endpoint cases are sharp. If $r=0$, then equality holds in \eqref{theorem_1_eq} for lattice tilings. For instance, take $K=[-\frac12,\frac12]^n$ with $\Lambda=\mathbb Z^n$, which yields $\mu(K,\Lambda)=1$, $\operatorname{vol}_n(K)=1$, and $\det(\Lambda)=1$.

If instead $r=n$, then equality holds in \eqref{theorem_1_eq} for the cross-polytope
$$
\begin{aligned}
K&=\operatorname{conv}\{\pm \boldsymbol e_1,\ldots,\pm \boldsymbol e_n\} \\
   &= \{\boldsymbol x\in\mathbb R^n:\|\boldsymbol x\|_1\le 1\}, 
\end{aligned}
$$
with $\Lambda=\mathbb Z^n$, where $\operatorname{conv}(A)$ denotes the convex hull of $A$, namely the intersection of all convex sets in $\mathbb R^n$ containing $A$. In particular, for this choice, we have 
$$
\lambda_1(K,\Lambda)=\cdots=\lambda_n(K,\Lambda)=1,
\qquad
\operatorname{vol}_n(K)=\frac{2^n}{n!},
\qquad
\det(\Lambda)=1.
$$
For completeness, the volume is obtained as follows. The body $K$ is the union of $2^n$ congruent simplices, one in each orthant. In the positive orthant, this simplex is
$$
\{\boldsymbol x\in\mathbb R_{\ge 0}^n:x_1+\cdots+x_n\le 1\},
$$
which has vertices $\{0, \boldsymbol e_1, \ldots, \boldsymbol e_n\}$ and volume $\frac{1}{n!}$. Because there are $2^n$ orthants in $\mathbb{R}^n$, it follows that the volume of $K$ is $\frac{2^n}{n!}$.


\subsection{A Sharp Intermediate Family}
Theorem~\ref{thm:main} is also sharp for intermediate values of $r$. Let $0<r<n$ and $T\ge 1$, and consider the body
$$
J_{n,r,T}
=
\left\{(\boldsymbol x, \boldsymbol y)\in\mathbb R^r\times\mathbb R^{n-r}:
\frac{\|\boldsymbol x\|_1}{T}+\|\boldsymbol y\|_\infty\le 1
\right\}
$$
with $\Lambda=\mathbb Z^n$ and
$L=\mathbb R^r\times\{\boldsymbol 0\}$.
The case $n=3$, $r=1$ and $T=3$ is illustrated in Figure~\ref{fig:J313}.
Notice that the sections of $J_{n,r,T}$ in the first $r$ coordinates are $r$-dimensional $\ell_1$-balls with radii decreasing linearly with $\|\boldsymbol y\|_\infty$, while its orthogonal projection onto 
$L^\perp=\{\boldsymbol 0\}\times\mathbb R^{n-r}$
corresponds with the cube $[-1,1]^{n-r}$.

\begin{figure}[ht!]
\centering
\includegraphics[width=0.65\textwidth]{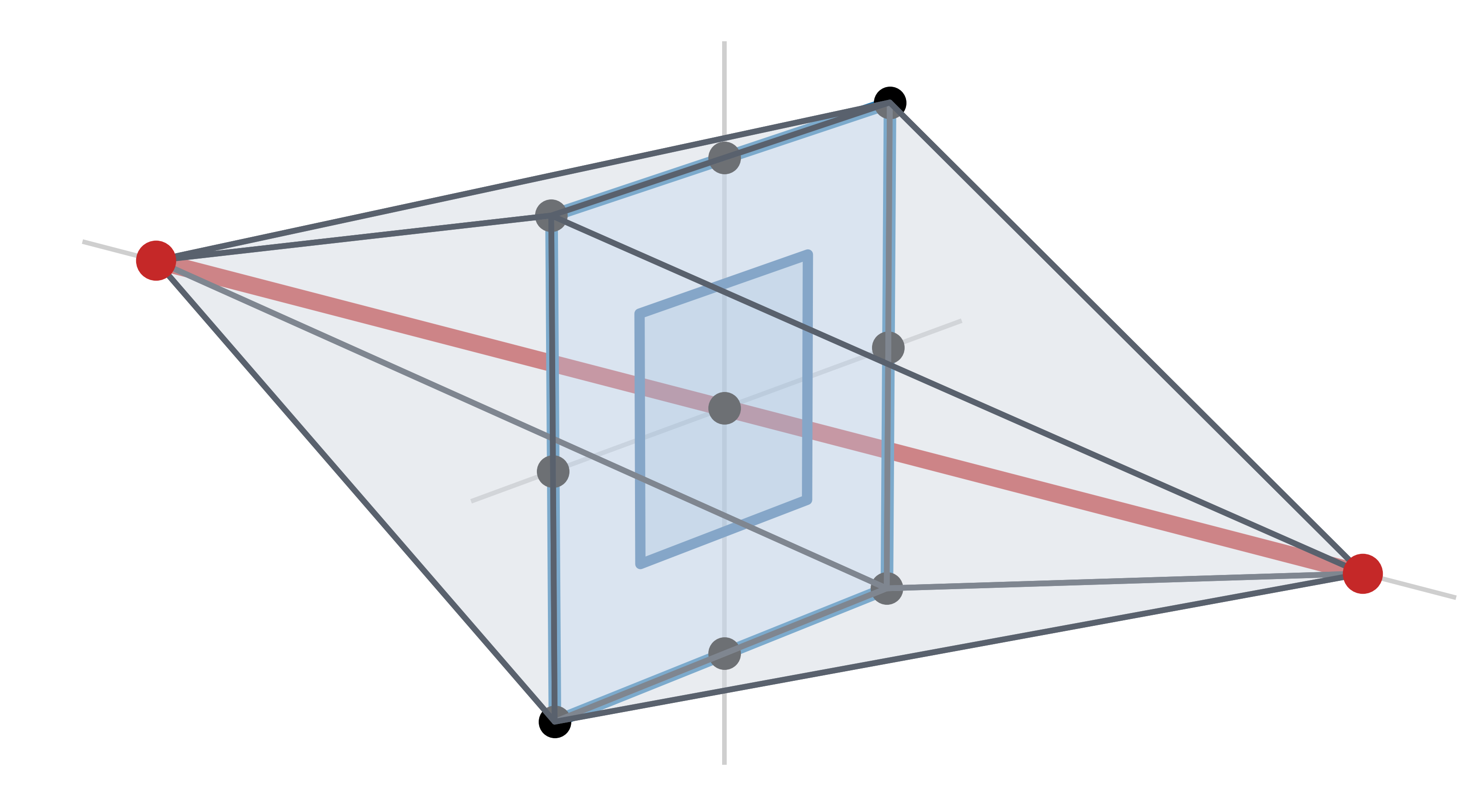}
\caption{This figure illustrates the convex body $J_{3,1,3}$. The red segment indicates the subspace $L$, the blue square is the projection $K_L=\pi_L (J_{3,1,3})$, the smaller (central) square is $\frac12K_L$, and the black points indicate the quotient lattice $\Lambda_L=\pi_L(\Lambda)$.}
\label{fig:J313}
\end{figure}

For this choice, the first $r$ successive minima are
$$
\lambda_1(J_{n,r,T},\mathbb Z^n)=\cdots=
\lambda_r(J_{n,r,T},\mathbb Z^n)=\frac1T.
$$
This follows since the standard/elementary basis vectors
$\boldsymbol e_1,\ldots,\boldsymbol e_r$ all belong to $(1/T)J_{n,r,T}$, and these $r$ vectors are linearly independent. Notice also that if $\lambda<1/T$, then the set $\lambda \, J_{n,r,T}$ contains no nonzero lattice points. Thus, the first $r$ successive minima are exactly $1/T$ in this case. 

The projection of $J_{n,r,T}$ onto
$L^\perp=\{\boldsymbol 0\}\times\mathbb R^{n-r}$ is
$$
K_L=\{\boldsymbol 0\}\times[-1,1]^{n-r},
$$
which corresponds with the cube $[-1,1]^{n-r}$ in $\mathbb R^{n-r}$. 
The projected lattice is
$$
\Lambda_L=\{\boldsymbol 0\}\times\mathbb Z^{n-r},
$$
which corresponds with $\mathbb Z^{n-r}$. Thus, it follows that
$$
\rho_L
=
\mu(K_L,\Lambda_L)
=
\mu([-1,1]^{n-r},\mathbb Z^{n-r})
=
\frac12.
$$

We next compute the volume of $J_{n,r,T}$. Note that for fixed $\boldsymbol y\in[-1,1]^{n-r}$, the corresponding section in the $\boldsymbol x$-variables is
$$
\left\{
\boldsymbol x\in\mathbb R^r:
\|\boldsymbol x\|_1\le T(1-\|\boldsymbol y\|_\infty)
\right\},
$$
which is an $r$-dimensional cross-polytope of volume 
$$
\frac{2^r}{r!}\bigl(T(1-\|y\|_\infty)\bigr)^r.
$$

Therefore, by Fubini's theorem (see e.g. \cite[Chapter 4]{weir1973lebesgue}), we have
$$
\operatorname{vol}_n(J_{n,r,T})
=
\int_{[-1,1]^{n-r}}
\frac{2^r}{r!}\bigl(T(1-\|\boldsymbol y\|_\infty)\bigr)^r\,d\boldsymbol y.
$$
The remaining integral satisfies
$$
\int_{[-1,1]^{n-r}}(1-\|\boldsymbol y\|_\infty)^r\,d\boldsymbol y
=
\frac{2^{n-r}(n-r)! \, r!}{n!},
$$
which follows by integrating over the level sets of $\|\boldsymbol y\|_\infty$.
Thus, it follows that
$$
\operatorname{vol}_n(J_{n,r,T})
=
\frac{2^rT^r}{r!}
\cdot
\frac{2^{n-r}(n-r)!r!}{n!}
=
\frac{2^nT^r(n-r)!}{n!}.
$$

Upon combining the values of the first $r$ successive minima with $\rho_L=1/2$, the left-hand side of \eqref{theorem_1_eq} is
$$
\left(\prod_{i=1}^r\lambda_i(J_{n,r,T},\mathbb Z^n)\right)
\rho_L^{\,n-r}
=
\left(\frac1T\right)^r
\left(\frac12\right)^{n-r}
=
\frac{1}{T^r2^{n-r}}.
$$
Recall that $\det(\mathbb Z^n)=1$, and notice that the right-hand side of \eqref{theorem_1_eq} is
$$
\frac{2^r(n-r)!}{n!}
\frac{\det(\mathbb Z^n)}{\operatorname{vol}_n(J_{n,r,T})}
=
\frac{2^r(n-r)!}{n!}
\frac{1}{{2^nT^r(n-r)!}/{n!}}
=
\frac{1}{T^r2^{n-r}},
$$
which shows that equality holds in Theorem~\ref{thm:main} for the family $J_{n,r,T}$ as claimed. 


\subsection{The Covering-Radius Form}
The above demonstrates that Theorem~\ref{thm:main} is sharp for $J_{n,r,T}$ whenever $0 < r < n$ and $T\ge1$. The corresponding covering-radius inequality, Corollary~\ref{cor:covering-radius}, is not sharp for fixed finite $T$, but becomes asymptotically sharp on this family in the limit $T\to\infty$.

Recall that
$$
J_{n,r,T}
=
\left\{(\boldsymbol x,\boldsymbol y)\in\mathbb R^r\times\mathbb R^{n-r}:
\frac{\|\boldsymbol x\|_1}{T}+\|\boldsymbol y\|_\infty\le 1
\right\}.
$$
The gauge associated with $J_{n,r,T}$ is
$$
\begin{aligned}
\|(\boldsymbol x,\boldsymbol y)\|_{J_{n,r,T}}
    &= \inf\{\lambda>0:(\boldsymbol x,\boldsymbol y)\in \lambda J_{n,r,T}\} \\
    &= \frac{\|\boldsymbol x\|_1}{T}+\|\boldsymbol y\|_\infty.
\end{aligned}
$$
It should be noted that this follows since 
$(\boldsymbol x,\boldsymbol y)\in \lambda J_{n,r,T}$ if and only if
$$
\frac{\|\boldsymbol x\|_1}{T}+\|\boldsymbol y\|_\infty\le \lambda,
$$
and hence the infimum is equal to
$$
\frac{\|\boldsymbol x\|_1}{T}+\|\boldsymbol y\|_\infty.
$$

We next compute the covering radius of $J_{n,r,T}$ with respect to $\mathbb Z^n$.  Recall that 
every point of $\mathbb R^n$ is congruent modulo $\mathbb Z^n$ to a point in the $o$-symmetric fundamental cube
$$
\left[-\frac 12, \frac 12\right]^n = \left[- \frac 12, \frac12 \right]^r\times\left[- \frac 12, \frac12\right]^{n-r}.
$$
It follows that the covering radius is obtained by maximising the $J_{n,r,T}$-gauge over this cube, namely 
$$
\mu(J_{n,r,T},\mathbb Z^n)
=
\max_{(\boldsymbol x,\boldsymbol y)\in\left[-\frac12,\frac12\right]^r\times\left[-\frac12,\frac12\right]^{n-r}}
\left(
\frac{\|\boldsymbol x\|_1}{T}+\|\boldsymbol y\|_\infty
\right).
$$

We now evaluate this maximum. For $\boldsymbol x\in[-\frac12,\frac12]^r$, we have
$$
\|\boldsymbol x\|_1\le \frac r2,
$$
with equality when $|x_1|=\cdots=|x_r|=\frac12$. For $\boldsymbol y\in[-\frac12,\frac12]^{n-r}$, we similarly have
$$
\|\boldsymbol y\|_\infty\le \frac12.
$$
It follows that 
$$
\begin{aligned}
\mu(J_{n,r,T},\mathbb Z^n)
    &= \frac{1}{T}\cdot\frac r2+\frac12 \\
    &=\frac12\left(1+\frac rT\right).
\end{aligned}
$$

It remains to compare the two sides of \eqref{corlloary_covering_radius_form} for this family. Recall that
$$
\lambda_1(J_{n,r,T},\mathbb Z^n)
=
\cdots
=
\lambda_r(J_{n,r,T},\mathbb Z^n)
=
\frac1T
\qquad 
\text{and} 
\qquad 
\operatorname{vol}_n(J_{n,r,T})
=
\frac{2^nT^r(n-r)!}{n!}.
$$
Since $\det(\mathbb Z^n)=1$, the right-hand side of \eqref{corlloary_covering_radius_form} is
$$
\frac{2^r(n-r)!}{n!}
\frac{1}{\operatorname{vol}_n(J_{n,r,T})}
=
\frac{1}{T^r2^{n-r}}.
$$
The left-hand side of \eqref{corlloary_covering_radius_form} is
$$
\left(\prod_{i=1}^r\lambda_i(J_{n,r,T},\mathbb Z^n)\right)
\mu(J_{n,r,T},\mathbb Z^n)^{n-r}
=
\left(\frac1T\right)^r
\left(\frac12\left(1+\frac rT\right)\right)^{n-r}.
$$
It follows that the ratio between the left-hand side and the right-hand side is
$$
\frac{
\left(\frac1T\right)^r
\left(\frac12\left(1+\frac rT\right)\right)^{n-r}
}{
{1}/{(T^r2^{n-r})}
}
=
\left(1+\frac rT\right)^{n-r},
$$
which tends to $1$ as $T\to\infty$. Thus, Corollary~\ref{cor:covering-radius} is asymptotically sharp on the family $J_{n,r,T}$ as claimed. 


\subsection{Comparison with the Jarník--Minkowski Estimate}
We also compare Corollary~\ref{cor:covering-radius} with what one obtains directly from Jarník's transference inequality and the lower bound in Minkowski's second theorem \eqref{minkowski_lower}. Fix $0\le r\le n$. Notice that the product
$$
\prod_{i=1}^r\lambda_i(K,\Lambda)
$$
already contains the first $r$ successive minima, so it remains to control the remaining factors
$$
\lambda_{r+1}(K,\Lambda),\ldots,\lambda_n(K,\Lambda)
$$
in terms of the covering radius.

Jarník's transference inequality (see e.g. \cite[Theorem~23.4]{gruber2007convex}) gives
$$
\mu(K,\Lambda)\ge \frac12\lambda_n(K,\Lambda).
$$
Recalling the ordering of the successive minima \eqref{ordering_successive}, we obtain
$$
\mu(K,\Lambda)\ge \frac12\lambda_i(K,\Lambda)
\qquad \text{for } i=r+1,\ldots,n.
$$
Upon multiplying over $i=r+1,\ldots,n$, we obtain
$$
\mu(K,\Lambda)^{n-r}
\ge
\frac{1}{2^{n-r}}
\prod_{i=r+1}^n\lambda_i(K,\Lambda), 
$$
and then multiplication by $\prod_{i=1}^r\lambda_i(K,\Lambda)$ gives
$$
\left(\prod_{i=1}^r\lambda_i(K,\Lambda)\right)
\mu(K,\Lambda)^{n-r}
\ge
\frac{1}{2^{n-r}}
\prod_{i=1}^n\lambda_i(K,\Lambda).
$$
Finally, using Minkowski's second theorem \eqref{minkowski_lower} yields
$$
\left(\prod_{i=1}^r\lambda_i(K,\Lambda)\right)
\mu(K,\Lambda)^{n-r}
\ge
\frac{1}{2^{n-r}}
\frac{2^n}{n!}
\frac{\det(\Lambda)}{\operatorname{vol}_n(K)}
=
\frac{2^r}{n!}
\frac{\det(\Lambda)}{\operatorname{vol}_n(K)}.
$$

Note that Corollary~\ref{cor:covering-radius} gives the stronger constant 
$$
\frac{2^r(n-r)!}{n!}
$$
on the right-hand side. Thus, Corollary~\ref{cor:covering-radius} improves the direct Jarník--Minkowski estimate by the factor $(n-r)!$, which is nontrivial when $n-r\ge 2$.


\bibliographystyle{plain}
\bibliography{references}

\end{document}